\documentclass[12pt]{amsart}
\usepackage[latin1]{inputenc}
\usepackage{color}
\usepackage{enumerate}
\usepackage{amssymb}
\usepackage[all]{xy}
\usepackage{hyperref}

\newtheorem{thm}{Theorem}[section]

\newtheorem{lem}[thm]{Lemma}

\theoremstyle{definition}
\newtheorem{defn}[thm]{Definition}
\theoremstyle{remark}
\newtheorem{rem}[thm]{Remark}

\newtheorem{example}[thm]{Example}
\numberwithin{equation}{section}
\newcommand{\norm}[1]{\left\Vert#1\right\Vert}
\newcommand{\abs}[1]{\left\vert#1\right\vert}
\newcommand{\set}[1]{\left\{#1\right\}}

\newcommand{\eps}{\varepsilon}
\newcommand{\To}{\longrightarrow}

\newcommand{\cA}{\mathcal A}
\newcommand{\cF}{\mathcal F}
\newcommand{\cG}{\mathcal G}

\newcommand{\fA}{\mathfrak A}

\usepackage{color}
\usepackage{enumerate}
\newcommand{\sub}{\subseteq}

\begin{document}

\title[Chain conditions in free Banach lattices]{Chain conditions in free Banach lattices}
\author[A.\ Aviles]{Antonio Avil\'es}
\address{Universidad de Murcia, Departamento de Matem\'{a}ticas, Campus de Espinardo 30100 Murcia, Spain.}
\email{avileslo@um.es}

\author[G.\ Plebanek]{Grzegorz Plebanek}
\address{Uniwersytet Wroc{\l}awski, Instytut Matematyczny, Pl. Grunwaldzki 2/4, 50-384 Wroclaw, Poland.}
\email{grzes@math.uni.wroc.pl}

\author[J.D. Rodr\'iguez Abell\'an]{Jos\'e David Rodr\'iguez Abell\'an}
\address{Universidad de Murcia, Departamento de Matem\'{a}ticas, Campus de Espinardo 30100 Murcia, Spain.}
\email{josedavid.rodriguez@um.es}

\subjclass[2010]{46B42}
\keywords{Free Banach lattice; countable chain condition; Knaster's condition; $sigma$-bounded chain condition; positively homogeneous continuous functions}
\thanks{
The research was done during the second author's stay at the University of Murcia,
supported by { Fundaci\'on S\'eneca} -- Agencia de Ciencia y Tecnolog\'{i}a de la Regi\'{o}n de Murcia,
 through its Re\-gio\-nal Programme {\em Jim\'enez de la Espada}. The first and third authors are supported by projects MTM2014-54182-P and MTM2017-86182-P (MINECO,AEI/FEDER, UE). The first author is also supported by 19275/PI/14 (Fundaci\'on S\'eneca).}
\begin{abstract}

We show that for an arbitrary Banach space $E$, 
the free Banach lattice $FBL[E]$ generated by $E$ satisfies the $\sigma$-bounded chain condition.

\end{abstract}

\maketitle

\section{Introduction}

The purpose of this paper is to investigate what chain conditions hold in  free Banach lattices generated by  Banach spaces.
 The concept of a Banach lattice freely generated by a given Banach space
  has been recently introduced and investigated by Avil\'es, Rodr\'iguez and Tradacete  \cite{ART17}. 
 
 Consider any Banach space $E$.
  Roughly speaking, the free Banach lattice generated by $E$ is a Banach lattice $F$ which contains a subspace linearly isometric with $E$ in such a way that its elements work as lattice-free generators. More formally,  the lattice $F$ has the following properties:

\begin{enumerate}[(i)]
\item there is a linear isometry $\phi:E\rightarrow F$ into its image;
\item for every Banach lattice $X$ and every bounded operator $T:E\rightarrow X$, there is a unique lattice homomorphism $\hat T:F\rightarrow X$ such that $\|\hat T\|=\|T\|$ and the following diagram commutes
$$\xymatrix{E\ar_{\phi}[d]\ar[rr]^T&&X\\
F\ar_{\hat{T}}[urr]&& }$$
\end{enumerate}

Those properties uniquely determine $F$ up to Banach lattice isometry, and so we can speak of  \emph{the} free Banach lattice generated by $E$, denoted by $FBL[E]$. This definition generalizes the notion of a free Banach lattice generated by a set $\Gamma$,  previously introduced by de Pagter and Wickstead \cite{dPW15}. Namely, the free Banach lattice generated by a set $\Gamma$ is the free Banach lattice generated by the Banach space $\ell_1(\Gamma)$, see Corollary 2.8 in \cite{ART17}.

Let us recall that the countable chain condition \emph{ccc} and  its various strengthenings,  typically considered in the context of Boolean algebras or topological spaces, in a more general setting
define combinatorial properties of partially ordered sets, see e.g.\ Todorcevic's survey article \cite{To00}.
Given a Banach lattice $X$, it is natural to discuss chain conditions of the partially ordered set $X^+$ of positive elements of the lattice. We shall  consider the following chain conditions formed in this way.

\begin{defn}\label{def}
We say that a Banach lattice $X$

\begin{enumerate}[(i)]
\item satisfies the {\em countable chain condition} (\emph{ccc}) if  for every uncountable family $\cF \subset X^+$ there are distinct $f,g\in \cF$ such that $f \wedge g \neq 0$;
\item satisfies {\em Knaster's condition} $K_2$ if  every uncountable family $\cF \subset X^+$  
 contains an uncountable family $\cG$ with the property that $f \wedge g \neq 0$ for every $f,g\in\cG$;
 \item satisfies the {\em $\sigma$-bounded chain condition ($\sigma$-bcc)}  if  $X^+$ admits a countable decomposition $X^+ = \bigcup_{n\ge 2} \cF_n$ such that, for every $n$,
 in every subset $\mathcal{G}\subset \cF_n$ of size $n$ there are two distinct elements $f,g\in \mathcal{G}$ such that $f\wedge g \neq 0$.
\end{enumerate}
\end{defn}

We have listed  those chain conditions  according to their increasing strength; in fact,   the implications
\[\sigma\mbox{\em -bcc}\Longrightarrow K_2\Longrightarrow  \mbox{\em ccc},\]
are valid for arbitrary partially ordered sets.
While it is clear that $K_2$ implies \emph{ccc}, the first implication is less obvious. Nonetheless,   the $\sigma$-bounded chain condition  does imply  $K_2$, as a consequence of the Dushnik-Miller partition theorem, cf.\
\cite[page 52]{To00}. We are grateful to Stevo Todorcevic for bringing this fact to our attention. The previous version of our paper contained  a separate argument that 
the lattice $FBL[E]$ satisfies Knaster's condition under an additional assumption that   $E$ is weakly compactly generated.
The role of the $\sigma$-bounded chain condition is briefly discussed in the final section below.

De Pagter and Wickstead \cite{dPW15} showed that the free Banach lattice $FBL[\ell_1(\Gamma)]$ generated by any set $\Gamma$ is always \emph{ccc}.
This is in analogy with the well-known property of  free Boolean algebras, which satisfy the countable chain condition regardless of their size (\cite[Chapter 4, Corollary 9.18]{Handbook}).
Assuming some linear and metric restrictions does not seem to help in constructing large sets of disjoint elements, and for this reason it is natural to guess that the free Banach lattice generated by any Banach space $E$ should also be \emph{ccc}. Although the original proof from \cite{dPW15} does not admit a straightforward generalization, we shall prove in this note that this is the case; in fact our main results reads as follows.

\begin{thm}\label{main1}
For every Banach space $E$,  the free Banach lattice $FBL[E]$ satisfies the  $\sigma$-bounded chain condition.
\end{thm}

We shall use the explicit description of $FBL[E]$  provided in \cite{ART17}. For a function $f:E^\ast \To \mathbb{R}$ define a norm

$$ \norm{f}_{FBL[E]} = \sup_{n \in \mathbb{N}} \set{\sum_{i = 1}^n \abs{f(x_{i}^{*})} : x_1^{*}, \ldots, x_n^{*} \in E^{*},\text{ }\sup_{x \in B_E} \sum_{i=1}^n \abs{x_i^{*}(x)} \leq 1 }.$$

We define $FBL[E]$ to be the closure of the vector lattice in $\mathbb R^{E^*}$ generated by the evaluations $\delta_x: x^\ast \mapsto x^\ast(x)$ with $x\in E$. These evaluations form the natural copy of $E$ inside $FBL[E]$. All the functions in $FBL[E]$ are positively homogeneous (that is, $f(rx^\ast) = r f(x^\ast)$ for all $x^\ast\in E^\ast$ and $r>0$) and $weak^\ast$-continuous  when restricted to the closed unit ball $B_{E^\ast}$. So, there is a natural inclusion
$$FBL[E] \subset C_{ph}(B_{E^\ast}),$$
where the right-hand side is the set of all weak$^\ast$-continuous and positively homogeneous functions on $B_{E^\ast}$. This inclusion preserves the order relation $\leq$ and the infimum and supremum operations $(\wedge,\vee)$, that are always defined pointwise.
Theorem \ref{main1} follows from Theorem~\ref{mainhomo} below, because the $\sigma$-bounded chain condition is transferred by the inclusion mentioned above.

\begin{thm}\label{mainhomo}
	The lattice $C_{ph}(B_{E^\ast})$ satisfies  the $\sigma$-bounded chain condition for every Banach space $E$.
\end{thm}


\section{The proof}

In the sequel, we often identify a natural number $n$ with the set $\{0,1,2,\ldots, n-1\}$.
For any set  $A$ and $s \in \mathbb{N}$, we use the following standard notation:
$[A]^s = \set{B \subset A : \abs{B} = s}$.

We start by recalling the classical Ramsey theorem \cite[Corollary 1.4]{To10} which we use in the proof of Lemma \ref{LemaRamsey}
below.

\begin{thm}(Ramsey)\label{Ramsey}
Given $p,\, q,\, r \in \mathbb{N}$, with $p \leq r$, there exists $N = N(p,q,r) \in \mathbb{N}$ such that for every map $$\psi: \left[N\right]^p \longrightarrow q$$ there exists $B \in \left[N\right]^r$ such that $\psi \vert_{[B]^p}$ is constant.
\end{thm}

Given any set $A$, we write $\Delta_A$ for the diagonal in $A\times A$.

\begin{lem} \label{LemaRamsey}
For every $a \in \mathbb{N}$, there exists $N = N(a) \in \mathbb{N}$ such that for every map
$$c: N \times N \setminus \Delta_N \longrightarrow a,$$
there exist $i<j<k \in N$ such that
$$c(i,j) = c(i,k) \text{ and } c(k,i) = c(k,j).$$
\end{lem}

\begin{proof}
We shall check that the Ramsey number  $N = N(3,a^2,5)$ given by Theorem \ref{Ramsey} has the required
property.
Fix any function $c: N \times N \setminus \Delta_N \longrightarrow a$.

Let $\varphi: \left[N\right]^3 \longrightarrow a^2$ be the map given by
$$\varphi(\{ i,j,k\}) = \left(c(i,j),c(k,j)\right),$$
whenever  $\{i,j,k\} \in \left[N\right]^3$ and $i<j<k$.
By Theorem \ref{Ramsey}, there exists $B \in \left[N\right]^5$ such that $\varphi$ is constant  on
${\left[B\right]^3}$. Write $B = \set{b_1, \ldots, b_5}$ so that   $b_1 < \ldots < b_5$.

We now check that   $b_2$, $b_3$,  $b_4$ is the triple satisfying the assertion of the lemma. Since
$$\varphi(\{b_2,b_3,b_4\}) = \varphi(\{b_2,b_4,b_5\}),$$
we get  $c(b_2,b_3) = c(b_2,b_4)$ by the definition of $\varphi$.

Analogously, since $$\varphi(\{b_1,b_3,b_4\}) = \varphi(\{b_1,b_2,b_4\}),$$
we conclude that $c(b_4,b_3) = c(b_4,b_2)$, and the proof is complete.
\end{proof}

Let us now fix a Banach space  $E$ and consider the compact space $K = (B_{E^*}, w^*)$.

\begin{thm} \label{phf1}
There is a countable decomposition $X = \bigcup_{\nu\in\mathbb{N}}X_\nu$ of the family $$X = \left\{ f\in C(K) : f\vert_{\frac{1}{3}K} \neq 0 \right\}$$
such that for every $\mathcal{G}\subset X_\nu$ of cardinality $\nu$ there exist two distinct $f, g \in \mathcal{G}$ such that $f \cdot g \neq 0$.
\end{thm}

\begin{proof}
What we are going to find is a countable decomposition $X=\bigcup_{w\in W} Y_w$, indicated on a suitable countable set $W$, together with a function $M:W\To \mathbb{N}$ such that for every $\mathcal{G}\subset Y_w$ of cardinality $M(w)$ there exist two distinct $f, g \in \mathcal{G}$ such that $f \cdot g \neq 0$. From such a decomposition we can define one like stated in the theorem, by picking either $X_\nu = \emptyset$ or $X_\nu = Y_w$ where $w$ is the least element (in some enumeration of $W$) that has not been previously chosen and $M(w)<\nu$.

Recall that for $x^*\in K$, sets of the form
$$V_{x^*}(x_1\ldots,x_n,\delta) = \set{y^* \in K : \abs{y^*(x_i)-x^*(x_i)}<\delta \text{ for every } i = 1,\ldots,n},
$$
where $x_1, \ldots, x_n \in E$ and $\delta > 0$, form a base for the $weak^*$ topology at $x^*\in K$.

For every $f\in X$ we have $f\vert_{\frac{1}{3}K} \neq 0$, so there is
$x_{f}^* \in E^*$ such that $\|x_{f}^*\|\le 1/3$ and $\abs{f(x_{f}^*)} > 0$. Without loss of generality we can assume that  there is  $\eps>0$ such that
$\abs{f(x_{f}^*)} > \eps$.

Every function $f\in X$ is $weak^*$ continuous at $x_f^*$ so there
is a $weak^*$ neighbourhood $U_f$ of $x_f^*$ such that
$\abs{f(y^*) - f(x_{f}^*)} < \eps/{2}$ for $y^*\in U_f$.
We may assume that every $U_f$ is a basic neighbourhood determined by $n_f$ vectors
from $E$ and some $\delta_f>0$ that can be supposed to be rational. Our index set will be $W = \mathbb{N}\times \mathbb{Q}$ and $Y(n,\delta) = \{f\in X : n_f = n, \delta_f = \delta\}$.

So we fix $w=(n\delta)$,  and what we have is that for every $f\in Y_w$ there exist $x_1^{f}, \ldots, x_n^{f} \in B_E$ satisfying

\begin{enumerate}[(i)]
\item $\abs{f(x_{f}^*)} \ge  \eps$ for every $f\in Y_w$;
\item writing $U_f=V_{x_{f}^*}(x_1^{f}, \ldots, x_n^{f}, \delta)$, we have
$\abs{f(y^*) - f(x_{f}^*)} < {\eps}/{2}$
 for every $y^* \in U_f$.
\end{enumerate}

In order to complete the proof it is enough to show that there is a large enough number $N$ (that will be our $M(w)$) that satisfies the following claim:
\medskip

\noindent {\sc Claim A.} Suppose that $\{f_0,\ldots,f_{N-1}\}\sub Y_w$.
There exist $0 \leq i < j < k \leq N-1$ such that for $y^{*} = x_i^* - x_j^* + x_k^* \in K$ we have  $f_i(y^*) \neq 0$ and $f_k(y^*) \neq 0$.
\medskip

Indeed, the general case follows then by reindexing the functions in question. In turn, Claim A follows from the following.
\medskip

\noindent {\sc Claim B.}
In the setting of Claim A, there are $0 \leq i < j < k \leq N-1$ such that
$$y^{*} = x_i^* - x_j^* + x_k^* \in U_i\cap U_k.$$

Indeed, if  $y^* \in U_i\cap U_k$ then

  $$\left.
 \begin{array}{rcl}
      \abs{f_i(x_i^*) - f_i(y^*)} & < & {\eps}/{2}
   \\ \abs{f_i(x_i^{*})} & > & \eps
 \end{array}
 \right\} \Rightarrow \abs{f_i(y^*)} > {\eps}/{2} \Rightarrow f_i(y^*) \neq 0,$$
and  $f_k(y^*) \neq 0$  for the same reason. To complete the proof we shall now verify Claim B.

Write $[-1,1] = \bigcup_{s=0}^{m-1} I_s$, where  $I_s$ are pairwise disjoint intervals (with or without endpoints)
 of diameter less than $\delta$.

The number $N=M(w)$ that we need to take is the number $N = N(m^n)$ given by Lemma \ref{LemaRamsey}. Consider
the mapping
$$c: N \times N \setminus \Delta_N \longrightarrow m^n,\quad  c(a,b) = \left(c_1(a,b), \ldots, c_{n}(a,b)\right),$$
 where for every $p\le  n$, the value of  $0 \leq c_p(a,b) \leq m-1$ is defined by the condition
 $x_b^*(x_p^{a}) \in I_{c_p(a,b)}$.

By Lemma \ref{LemaRamsey}, there exist $i < j < k \leq N-1$ such that $$c(i,j) = c(i,k) \text{ and } c(k,i) = c(k,j).$$

As $c(i,j) = c(i,k)$, for every $p\le n$ we have  $ \abs{x_j^*(x_p^{i}) - x_k^*(x_p^{i})} < \delta$, and
$$\abs{x_i^{*}(x_p^{i}) - y^*(x_p^{i})} = \abs{x_j^{*}(x_p^{i}) - x_k^*(x_p^{i})} < \delta,$$
which means that $ y^* \in V_{x_i^*}(x_1^{i}, \ldots, x_n^{i}, \delta)=U_i$.

In the same manner,  from $c(k,i) = c(k,j)$ we get  $\abs{x_i^*(x_p^{k}) - x_j^*(x_p^{k})} < \delta$, and
$$\abs{x_k^{*}(x_p^{k}) - y^*(x_p^{k})} = \abs{x_i^{*}(x_p^{k}) - x_j^*(x_p^{k})} < \delta.$$
Again, this means that  $ y^* \in V_{x_k^*}(x_1^{k}, \ldots, x_n^{k}, \delta)=U_k$, and this verifies Claim B.
\end{proof}

Theorem~\ref{mainhomo} follows immediately from Theorem \ref{phf1}, because all positive elements of $C_{ph}(K)$ satisfy $f\vert_{\frac{1}{3}K} \neq 0$. As it was observed in the introduction, Theorem~\ref{main1} follows from Theorem~\ref{mainhomo}.\\

\section{Concluding remarks}

 Let us note that  the proof of Theorem~\ref{phf1} works even if we replace  $B_{E^\ast}$ by any
  weak$^\ast$-closed and absolutely convex subset $K$ of $B_{E^\ast}$. 
  The only delicate point in the proof that one has to be careful about is that  the vector $y^{*} = x_i^* - x_j^* + x_k^*$ chosen in Claim B  is still an element of $K$ and this
  is guaranteed by  
   $x_i^*, x_j^*,x_k^*\in\frac{1}{3}K$. Thus, Theorem~\ref{mainhomo} may be stated as follows.
    
 \begin{thm}
Given a  Banach space $E$ and a weak$^\ast$-closed  absolutely convex set $K\subset B_{E^\ast}$, the lattice
$C_{ph}(K)$ satisfies  the $\sigma$-bounded chain condition.
 \end{thm}

The question arises if there are natural stronger chain conditions that would hold in $C_{ph}(B_{E^*})$, and so in $FBL[E]$, 
for every Banach space $E$.

The $\sigma$-bounded chain condition was introduced by Horn and Tarski in connection with Boolean algebras carrying a strictly positive measures.
It is worth recalling that the related Horn-Tarski problem, whether the condition $\sigma$-{\em bcc} is equivalent to its certain formally weaker version was solved in the negative
only a few years ago by Th\"{u}mmel \cite{Th14} and Todorcevic \cite{To14}.
 
Suppose that $\fA$ is a Boolean algebra and $\mu:\fA\to [0,1]$ is a finitely additive probability measure such that $\mu(a)>0$ for every $a\in\fA^+$.
Then we can write 
\[ \fA^+=\bigcup_{n\ge 2} \cF_n, \mbox{ where } \cF_n=\{a\in\fA: \mu(a)>1/n\}.\] 
Clearly, 
$\cF_n$ contains no $n$ many pairwise disjoint elements, so $\fA$ satisfies the $\sigma$-bounded chain condition. 
This cannot be reversed,  there are algebras with $\sigma${\em -bcc} not carrying strictly positive measures; cf.\ Chapter 6 of \cite{CN82}.

If $X$ is a sublattice of the space $C(K)$ for some compact space $K$ then one can think of  an analogous chain-like condition, stating that there is a finitely additive probability  measure $\mu$ on $K$
which is {\em strictly positive on} $X^+$, that is  $\int_K f\;{\rm d}\mu>0$ for every $f\in X^+$. 
Note that to have $\int_K f\;{\rm d}\mu$ well-defined for every continuous function $f$ we need only  to assume that   the domain of $\mu$ contains the algebra $\cA(K)$ generated by
closed subsets of $K$.  
Once we have such $\mu$, it is not difficult to verify the condition $\sigma$-{\em bcc}. 
Let us first observe that 
whether the measure in question is actually countably additive or merely finitely additive is not essential here.

\begin{rem}\label{remark}
Suppose that $\mu$ is finitely additive
probability measure which is strictly positive on $X^+$ for some sublattice $X$ of the lattice $C(K)$ of continuous functions on a compact space $K$.
Then there is a countably additive Borel measure $\mu'$ on $K$ which is again strictly positive on $X^+$.

For $f\in X^+$ write  $\eps=\int_K f\;{\rm d}\mu$ and $A=\{x\in K : f(x)\ge \eps/2\}$; then
\[ \eps=\int_{A}f\;{\rm d}\mu+\int_{K\setminus A}f\;{\rm d}\mu\le \|f\|_\infty\cdot\mu(A)+\eps/2,\]
which gives $\mu(A)>0$. This implies that whenever  a finitely additive measure ${\mu}'$ satisfies 
${\mu}'(A)\ge\mu(A)$ for every closed $A\subset K$ then again $\int_K f\;{\rm d}\mu'>0$ for every $f\in X^+$.
Now the point is that there is such $\mu'$ that is closed-inner-regular on the algebra  $\cA(K)$, see \cite{Pl86};
$\mu'$ is then countably additive (by compactness) and, consequently,  extends to a countable additive Borel measure on $K$ which is
positive on $X^+$.
\end{rem}

Using Remark \ref{remark} it is not difficult to give an example showing that the $\sigma$-bounded chain condition that holds in every  $FBL[E]$ does not admit the obvious measure-theoretic strengthening mentioned above.  

\begin{example}
Consider the Banach space $E=c_0(\Gamma)$, where $\Gamma$ is a uncountable set; then $E^*=\ell_1(\Gamma)$. 
There is no measure on $K=B_{E^*}$ which would be positive on all elements from $C_{ph}(K)^+$. 

Indeed,  every $\gamma\in\Gamma$ defines $f_\gamma\in C_{ph}(K)^+$, where $f_\gamma(x)=|x_\gamma|$. Suppose that $\mu$ is a measure on $K$
such that $\int_K  f_\gamma\;{\rm d}\mu>0$ for every $\gamma$. By Remark \ref{remark} we can assume that $\mu$ is countable additive.
Then for every $\gamma$ there is $\delta(\gamma)>0$ such that 
\[\mu\left(\{x\in K: f_\gamma(x)\ge\delta(\gamma)\}\right)>0.\]
Using the fact  that $\Gamma$ is unountable, we conclude easily that  there is $\delta>0$ and a sequence of distinct $\gamma_n\in\Gamma$ such that,
writing $A_n=\{x\in K: f_{\gamma_n}(x)\ge\delta\}$, we have $\mu(A_n)\ge\delta$. But then $\mu(\bigcap_n\bigcup_{k\ge n} A_k)\ge\delta$; in particular,
there is $x\in K$ belonging to infinitely many sets $A_n$. This clearly contradicts the fact that $x\in K\sub \ell_1(\gamma)$.
\end{example}

\end{document}